\documentclass{elsarticle}

\usepackage{amsthm, amsmath, amssymb, latexsym}
\usepackage{graphicx}
\usepackage[update,prepend]{epstopdf}

\usepackage{color}

\newcommand{\mR}{\mathbb R}

\theoremstyle{plain}
\newtheorem{theorem}{Theorem}

\newtheorem{corollary}{Corollary}

\theoremstyle{definition}

\newtheorem{assumption}{Assumption}

\theoremstyle{remark}
\newtheorem{remark}{Remark}

\DeclareMathOperator{\sign}{sign}

\begin{document}

\begin{frontmatter}

\title{Observer-based boundary control of the sine-Gordon model energy}
\tnotetext[mytitlenote]{This work was performed in the IPME RAS
and supported by the Russian Science Foundation (grant No 14-29-00142).}

\author[First,Third]{Maksim Dolgopolik\corref{mycorrespondingauthor}} 
\author[First,Second,Third]{Alexander L. Fradkov} 
\author[First,Second,Third]{Boris Andrievsky}

\address[First]{Institute of Problems of Mechanical Engineering, 
   Saint Petersburg, Russia (e-mails: maxim.dolgopolik@gmail.com,
   Alexander.Fradkov@gmail.com, boris.andrievsky@gmail.com).}
\address[Second]{Saint Petersburg State University, 
   Saint Petersburg, Russia}
\address[Third]{University ITMO, 
  Saint Petersburg, Russia}

\cortext[mycorrespondingauthor]{Corresponding author}

\begin{abstract}
In this paper the output feedback energy control problem for the sine-Gordon model is studied. An observer for the
sine-Gordon equation and a speed-gradient boundary control law for solving this problem are analysed.
Explicit inequalities on system's parameters ensuring the exponential decay of the estimation error are obtained.
Under an additional assumption the achievement of the control goal is proved. The results of numerical experiments
demonstrate that the transient time in energy is close to the transient time in observation error.
\end{abstract}

\begin{keyword}
boundary control \sep energy control \sep speed-gradient \sep sine-Gordon equation
\MSC[2010] 93C20 \sep 35L71
\end{keyword}

\end{frontmatter}

\section{Introduction}

This paper is devoted to control of oscillations in nonlinear distributed systems. As the example the celebrated
sine-Gordon model was chosen. The sine-Gordon equation is a semilinear wave equation used to model many physical
phenomena like Josephson junctions, seismic events including earthquakes, slow slip and after-slip processes,
dislocation in solids, etc.\cite{Cuevas14}. Though control of oscillatory modes is one of conventional areas of control
theory, most works in the past were dealing with linear models and regulation or tracking as control objectives. It was 
motivated by practical problems such as vibration suppression, vibration isolation, etc. \cite{Fuller96}. Some of the
methods were extended to distributed systems and  allowed one to stabilize wave motion
\cite{Zuazua90,Lasiecka92,Morgul94}. By the beginning of the new century quite a number of control methods for
distributed (PDE) systems were proposed for the regulation and tracking problems including optimal control, robust
control, adaptive control, etc. \cite{Troltzsch,Christofides01,Smyshlyaev10}. In a number of publications a powerful
backstepping method was developed (see \cite{Krstic08} and the references therein).

In the 1990s an interest in new problems, related to oscillations shaping and synchronization rather than their
suppression  was growing \cite{YangChua97,FP98,Pikovsky01}. An efficient  approach to such problems based on the system
speed-gradient and energy control was developed in \cite{F96,Shiriaev00}. Energy is a fundamental characteristics of a
physical system which is conserved if the losses can be neglected. Changing the desired level of energy allows one to
specify various desired properties of oscillations and describe various system modes. However an approach to energy
control proposed in \cite{F96,Shiriaev00} based on the speed-gradient method was not extended to distributed systems and
to control of waves until recent. The reason is in that the desired energy level set may be a complex unbounded set with
complex geometry. 

Perhaps the first attempt to study energy control approach for distributed systems was made in \cite{Porubov15} where  a
possibility of controlling travelling waves in the sine-Gordon model was demonstrated by simulation. In \cite{Orlov17}
two energy  control algorithms were studied rigorously: one provides spatially distributed control action while another
one provides a uniform over the space one. Both algorithms, however have some drawbacks. To improve the proposed
approach it would be desirable to study the boundary control problem.

The first boundary control algorithms for the sine-Gordon equation were proposed in
\cite{Kobayashi03,Petcu04,Kobayashi04}. In these papers, however, only the stabilization problem was considered. In
\cite{DFA16} the boundary control of the sine-Gordon system energy was studied. However the solution proposed in
\cite{DFA16} requires measurements of the system velocity which may be not available apart from the boundary of the
domain.

In this paper the observer-based solution to the problem  of the sine-Gordon model energy boundary control is proposed.
The Luenberger-type observer to evaluate velocities required for evaluation of the system energy is proposed. To design
the energy control algorithm the speed-gradient method \cite{FP98} is employed. To analyze the system well-posedness and
partial stability an energy-based Lyapunov functional is used. The closed loop dynamics are illustrated by simulation.

The paper is structured as follows. In Section~\ref{Sec_ProblemFormulation} the problem is formulated and necessary
notations and definitions are introduced. Control algorithm is designed in Section~\ref{Sec_ControlAlgorithm}.
Section~\ref{Sec_MainTheorems} is devoted to studying the properties of the designed control system. Numerical evaluation results
are presented in Section~\ref{Sec_Simulation}.

\section{Problem Formulation}
\label{Sec_ProblemFormulation}

Consider the one-dimensional sine-Gordon equation with the following initial and boundary conditions
\begin{gather} \label{SineGordon}
  z_{tt}(t, x) - k z_{xx}(t, x) + \beta \sin z(t, x) = 0, \quad t \ge 0, \quad x \in (0, 1) \\
  z(0, x) = z^0(x), \quad z_t(0, x) = z^1(x), \quad x \in (0, 1), \label{SineGordon1} \\
  z(t, 0) = 0, \quad  z_x(t, 1) = u(t), \quad t \ge 0, \label{SineGordon2} \\
  y(t) = z_t(t, 1) \quad t \ge 0, \label{SineGordon3}
\end{gather}
where $\beta > 0$ and $k > 0$ are given parameters, $u(t)$ is a control input, $y(t)$ is the output, 
and $z^0, z^1 \colon [0, 1] \to \mathbb{R}$ are given functions. Denote by
$$
  H(z) = \int_0^1 \left( \frac{z_t^2}{2} + k \frac{z_x^2}{2} + \beta (1 - \cos z) \right) \, dx
$$
the Hamiltonian for the sine-Gordon equation. One can easily verify that the Hamiltonian $H(z)$ is preserved
along solutions of the unforced system. Furthermore, $H(z)$ is a non-negative function, and $H(z) = 0$ if and only
if $z = 0$. Therefore $H(z(t))$ can be viewed as \textit{the energy} of a solution $z(t, x)$ of equation
\eqref{SineGordon} (or \textit{system's energy}) at time $t$.

We pose the following control problem: find a control law $u(t)$, which ensures the control objective
\begin{equation} \label{ControlGoal}
  H(z(t)) \to H^* \text{~as~} t \to + \infty,
\end{equation}
where, $z(t)$ is a solution of \eqref{SineGordon}--\eqref{SineGordon3}, and $H^* \ge 0$ is prespecified. Thus, 
the control objective is to reach the desired energy level $H^*$ in the system~\eqref{SineGordon}--\eqref{SineGordon3}.

\begin{remark}
It should be noted that all results below can be easily extended to the case
$$
  z_t(t, 1) = u(t), \quad y(t) = z_x(t, 1) \quad t \ge 0,
$$
(i.e. one can swap the input and the output), since the derivatives $z_t(t, 1)$ and $z_x(t, 1)$ enter all expressions
below almost identically with the only difference being the coefficient $k > 0$ before $z_x(t, 1)$.
\end{remark}

\section{Speed-Gradient Control Law and Luenberger Type Observer}
\label{Sec_ControlAlgorithm}

Let us utilize the Speed-Gradient algorithm in order to design a control law solving the boundary energy control problem
posed above. Introduce the goal function
$$
  Q(z) = \frac{1}{2} \big( H(z) - H^* \big)
$$
that measures the difference between the current and the desired energies. The derivative of this function along
solutions of the system \eqref{SineGordon}--\eqref{SineGordon3} has the form
$$
  \frac{d}{dt} Q(z(t)) = (H(z(t)) - H^*) \int_0^1 \Big( z_t z_{tt} + k z_x z_{xt} + \beta \sin z z_t \Big) \, dx.
$$
Substituting $z_{tt}$ for $k z_{xx} - \beta \sin z$ (recall that $z(t)$ is a solution of \eqref{SineGordon}), and
integrating the term $z_x z_{xt}$ by parts one obtains
\begin{multline*}
  \frac{d}{dt} Q(z(t)) = (H(z(t)) - H^*)
  \Bigg[ \int_0^1 \big( z_t (k z_{xx} - \beta \sin z) - k z_{xx} z_{t} + \beta \sin z z_t \big) \, dx \\
  + k z_x(t, 1) z_t(t, 1) - k z_x(t, 0) z_t(t, 0) \Bigg].
\end{multline*}
Hence with the use of the boundary condition $z(t, 0) = 0$ one finally gets that
$$
  \frac{d}{dt} Q(z(t)) = k (H(z(t)) - H^*) u(t) y(t)
$$
Then according to the Speed-Gradient algorithm one defines the control law as follows:
$$
  u_0(t) = - \gamma \frac{\partial}{\partial u} \frac{d Q(z)}{dt} = - \gamma (H(z(t)) - H^*) k y(t),
$$
where $\gamma > 0$ is a scalar gain. Below, we consider the more general control algorithm of the form
\begin{equation}	\label{PreliminaryControlLaw}
  u_0(t) = - \gamma \psi(H(z(t)) - H^*) y(t),
\end{equation}
where $\psi \colon \mathbb{R} \to \mathbb{R}$ is a continuous function such that $\psi(0) = 0$ and $\psi(s) s > 0$ for
any $s \ne 0$. 

Note that control law \eqref{PreliminaryControlLaw} defined with the use of the Speed-Gradient algorithm depends on the
current energy of the system $H(z(t))$ that is not available from the measurements. Therefore we need to design an
observer. Being inspired by the ideas of \cite{Fridman2016}, we propose a Luenberger-type observer of the form:
\begin{gather} \label{Observer}
  \widehat{z}_{tt}(t, x) - k \widehat{z}_{xx}(t, x) + \beta \sin \widehat{z}(t, x) = 0, 
  \quad t \ge 0, \quad x \in (0, 1) \\
  \widehat{z}(0, x) = \widehat{z}^0(x), \quad \widehat{z}_t(0, x) = \widehat{z}^1(x), 
  \quad x \in (0, 1), \label{Observer1} \\
  \widehat{z}(t, 0) = 0, \quad  
  \widehat{z}_x(t, 1) = u(t) + \alpha\big( y(t) - \widehat{z}_t(t, 1) \big), 
  \quad t \ge 0, \label{Observer2}
\end{gather}
where $\widehat{z}^0, \widehat{z}^1 \colon [0, 1] \to \mathbb{R}$ are given functions, and $\alpha > 0$ is an observer
gain. Now, with the use of the above observer we define the control law as follows:
\begin{equation} \label{ControlLaw}
  u(t) = - \gamma \psi\Big( H(\widehat{z}(t)) - H^* \Big) y(t).
\end{equation}
In the following section, we demonstrate that under some additional assumptions control law \eqref{ControlLaw} solves
the energy control problem \eqref{ControlGoal}.

\begin{remark}
Let us note that the observer \eqref{Observer}--\eqref{Observer2} can be designed with the use of the Speed-Gradient
algorithm as well. Namely, consider the model of the system \eqref{SineGordon}--\eqref{SineGordon3} of the form
\begin{gather*} 
  \widehat{z}_{tt}(t, x) - k \widehat{z}_{xx}(t, x) + \beta \sin \widehat{z}(t, x) = 0, 
  \quad t \ge 0, \quad x \in (0, 1) \\
  \widehat{z}(0, x) = \widehat{z}^0(x), \quad \widehat{z}_t(0, x) = \widehat{z}^1(x), 
  \quad x \in (0, 1), \\
  \widehat{z}(t, 0) = 0, \quad  
  \widehat{z}_x(t, 1) = u(t) + \widehat{u}(t), \quad t \ge 0, \\
  \widehat{y}(t) = \widehat{z}_t(t, 1).
  \quad t \ge 0,
\end{gather*}
where $\widehat{u}(t)$ is a control input, and $\widehat{y}(t)$ is the output of the model. Define the goal function 
$$
  \widehat{Q}(\widehat{z}) = \int_0^1 
  \left( \frac{(z_t - \widehat{z}_t)^2}{2} + k \frac{(z_x - \widehat{z}_x)^2}{2} \right) \, dx,
$$
i.e. the goal function is a weighted estimation error. Then one can easily check that applying the Speed-Gradient
algorithm in this case we arrive at the control law $\widehat{u}(t) = \alpha (y(t) - \widehat{y}(t))$, which coincides
with the one used in \eqref{Observer2}.
\end{remark}

\section{Properties of the Designed Control System}
\label{Sec_MainTheorems}

Let us study the performance of the system \eqref{SineGordon}--\eqref{SineGordon3} with control law \eqref{ControlLaw}
and observer \eqref{Observer}--\eqref{Observer2}. The closed-loop system has the from
\begin{gather} \label{CLSyst_1}
  z_{tt}(t, x) - k z_{xx}(t, x) + \beta \sin z(t, x) = 0, \quad t \ge 0, \quad x \in (0, 1) \\
  \widehat{z}_{tt}(t, x) - k \widehat{z}_{xx}(t, x) + \beta \sin \widehat{z}(t, x) = 0, 
  \quad t \ge 0, \quad x \in (0, 1) \\
  z(0, x) = z^0(x), \quad z_t(0, x) = z^1(x), \quad x \in (0, 1), \\
  \widehat{z}(0, x) = \widehat{z}^0(x), \quad \widehat{z}_t(0, x) = \widehat{z}^1(x), 
  \quad x \in (0, 1), \\
  z(t, 0) = 0, \quad  z_x(t, 1) = u(t), \quad t \ge 0, \label{CLSyst_BoundaryCond} \\
  \widehat{z}(t, 0) = 0, \quad  
  \widehat{z}_x(t, 1) = u(t) + \alpha\big( z_t(t, 1) - \widehat{z}_t(t, 1) \big), 
  \quad t \ge 0, \\
  u(t) = - \gamma \psi\Big( H(\widehat{z}(t)) - H^* \Big) z_t(t, 1) 
  \quad t \ge 0. \label{CLSyst_2} 
\end{gather}
Let us introduce an assumption on the well-posedness of this system.

\begin{assumption} \label{Assumpt_WellPosedness}
There exists a nonempty set $\mathcal{W}_0 \subseteq W^2_2(0, 1) \times W^1_2(0, 1)$ of ``sufficiently smooth'' initial
data such that for any $(z^0, z^1) \in \mathcal{W}_0$ and $(\widehat{z}^0, \widehat{z}^1) \in \mathcal{W}_0$ there
exists a unique ``sufficiently regular'' solution $(z(t, x), \widehat{z}(t, x))$ of the system
\eqref{CLSyst_1}--\eqref{CLSyst_2} such that
\begin{enumerate}
\item{$(z(t), \widehat{z}(t))$ is defined on a maximal interval of existence $[0, T_{\max})$, and if
$T_{\max} < + \infty$, then $H(z(t)) + H(\widehat{z}(t)) \to + \infty$ as $t \to T_{\max}$,
}

\item{the functions $H(z(\cdot))$ and $H(\widehat{z}(\cdot))$ are locally absolutely continuous.
}
\end{enumerate}
\end{assumption}

\begin{remark}
Note that for the main results of this paper to hold true it is sufficient (but not necessary) to suppose that 
$$
  z, \widehat{z} \in C\Big( [0, T_{\max}); W^2_2(0, 1) \Big) \cap C^1\Big( [0, T_{\max}); W^1_2(0, 1) \Big) \cap
  C^2\Big( [0, T_{\max}); L_2(0, 1) \Big),
$$
where, as in the assumption above, $W^m_p(0, 1)$ is the Sobolev space.
\end{remark}

\begin{remark}
Let us point out the difficulties in the proof of existence and uniqueness theorem for the initial-boundary value
problem \eqref{CLSyst_1}--\eqref{CLSyst_2}. At first, note that this problem cannot be rewritten (without some
nontrivial trasformations) as a Lipschitz perturbation of a linear evolution equation of the form
$$
  \frac{d w}{dt}  = \mathcal{A} w + f(w),
$$
where $\mathcal{A}$ is an unbounded linear operator in a Banach space $X$. The interested reader can check that
regardless of the choice of the space $X$ either the operator $\mathcal{A}$ is not an infinitesimal generator of a
$C_0$-semigroup or the nonlinear operator $f$ is not locally Lipschitz continuous or $f$ is not defined on the entire
space $X$ (or not in the domain of $\mathcal{A}$, so that such results as \cite{Pazy}, Theorem~6.1.7 are not
applicable). One can consider the problem \eqref{CLSyst_1}--\eqref{CLSyst_2} as a nonlinear evolution equation, but
the corresponding nonlinear operator does not possess any standard properties. It is neither accretive (dissipative) nor
compact. Similarly, all other general methods for proving the existence of solutions of nonlinear hyperbolic partial
differential equations known to the authors cannot be directly applied to the problem under consideration. Therefore we
pose the above assumption on the well-posedness of the system \eqref{CLSyst_1}--\eqref{CLSyst_2} as a challenging
problem for future research. It should be noted that the main difficulty in a proof of this assumption consists in the
fact that dynamic boundary conditions \eqref{CLSyst_BoundaryCond}--\eqref{CLSyst_2} are nondissipative and
\textit{nonlinearly} depend on the derivative $z_t(t, 1)$.
\end{remark}

\subsection{Exponential Decay of the Estimation Error}

We start our analysis of the control law \eqref{ControlLaw} by showing that under some additional assumptions 
the equation for the estimation error $e = z - \widehat{z}$ is globally exponentially stable.

Observe that the function $e(t, x)$ is a solution of the following boundary value problem:
\begin{gather} \label{EstimError}
  e_{tt}(t, x) - k e_{xx}(t, x) + \beta \Big( \sin z(t, x) - \sin \widehat{z}(t, x) \Big) = 0, \\
  e(0, x) = z^0(x) - \widehat{z}^0(x), \quad e_t(0, x) = z^1(x) - \widehat{z}^1(x), \label{EstimError1} \\
  e(t, 0) = 0, \quad  e_x(t, 1) = - \alpha e_t(t, 1), \label{EstimError2}
\end{gather}
where $t \ge 0$ and $x \in [0, 1]$. Denote by
$$
  E(t) = \frac{1}{2} \int_0^1 \Big( e_t^2 + k e_x^2 \Big) \, dx
$$
the weighted quadratic error. Our aim is to show that $E(t)$ decays exponentially, provided the parameters $k$ and
$\beta$ satisfy certain conditions. In order to conveniently express these conditions, denote
\begin{equation} \label{EtaDef}
  \eta(\beta, k) = \max\left\{ \beta, \frac{4 \beta}{\pi^2 k - (\pi^2 + 4) \beta}\right\},
\end{equation}
and introduce the following assumption.

\begin{assumption} \label{Assumpt_AdmissibleParam}
The parameters $k > 0$ and $\beta > 0$ of the system \eqref{SineGordon}--\eqref{SineGordon2} satisfy the following
inequalities:
$$
  \left( 1 + \frac{4}{\pi^2} \right) \beta < k, \quad \eta(\beta, k) < \min\big\{  1, k \big\}.
$$
\end{assumption}

One can easily verify that assumption~\ref{Assumpt_AdmissibleParam} is valid iff
$$
  \begin{cases}
    \beta < 1, & \text{in the case } k > \dfrac{\pi^2 + 8}{\pi^2} \approx 1.81, \\
    \beta < \dfrac{\pi^2 k}{\pi^2 + 8}, & \text{in the case } 1 \le k \le \dfrac{\pi^2 + 8}{\pi^2}, \\
    \beta < \dfrac{\pi^2 k^2}{(\pi^2 + 4)k + 4}, & \text{in the case } 0 < k < 1.
  \end{cases}
$$
Note that for assumption~\ref{Assumpt_AdmissibleParam} to hold true it is necessary that $\beta < 1$.

\begin{theorem} \label{Th_Observer}
Let Assumptions \ref{Assumpt_WellPosedness} and \ref{Assumpt_AdmissibleParam} be valid. Then for any $\varepsilon > 0$
such that $\eta(\beta, k) < \varepsilon < \min\{ 1, k \}$, and for all $\alpha > 0$ satisfying the inequality
\begin{equation} \label{AdmissibleObserverGain}
  \frac{\varepsilon k}{2} \alpha^2 - k \alpha + \frac{\varepsilon}{2} \le 0
\end{equation}
there exist $\delta > 0$ and $M > 0$ such that for any initial conditions $(z^0, z^1) \in \mathcal{W}_0$ and
$(\widehat{z}^0, \widehat{z}^1) \in \mathcal{W}_0$, and for all $\gamma > 0$ one has 
$$
  E(t) \le M E(0) e^{-\delta t} \quad \forall t \in [0, T_{\max}).
$$
\end{theorem}

\begin{proof}
For any $\varepsilon > 0$ introduce the Lyapunov function
$$
  V(t) = E(t) + \varepsilon \int_0^1 x e_t e_x \, dx.
$$
(cf.~\cite{Kobayashi04}, Section~4; \cite{Fridman2016}, Section~2.2). From the inequalities
$$
  \Big| \int_0^1 x e_t e_x \, dx \Big| \le \int_0^1 |e_t| |e_x| \, dx \le
  \frac{1}{2} \int_0^1 e_t^2 \, dx + \frac{1}{2} \int_0^1 e_x^2 \, dx \le 
  \max\left\{ 1, \frac{1}{k} \right\} E(t),
$$
it follows that for any $\varepsilon < \min\{ 1, k \}$ one has
\begin{equation} \label{ErrorLyapFuncEstimates}
  0 \le (1 - k_0 \varepsilon) E(t) \le V(t) \le (1 + k_0 \varepsilon) E(t), \quad
  (1 - k_0 \varepsilon) > 0,
\end{equation}
where $k_0 = \max\{ 1, 1 / k \}$. Thus, in particular, for any $0 < \varepsilon < \min\{ 1, k \}$ one has 
$V(t) \ge 0$.

For any $t \in [0, T_{\max})$ one has
\begin{multline} \label{ErrorLyapFuncDer}
  \frac{d}{dt} V(t) = \frac{d}{dt} E(t) + \varepsilon \frac{d}{dt} \int_0^1 x e_t e_x \, dx \\
  = \int_0^1 \Big( e_t e_{tt} + k e_{x} e_{tx} \Big) \, dx +
  \varepsilon \int_0^1 \Big( x (e_{tt} e_x + e_t e_{tx}) \Big) \, dx.
\end{multline}
Taking into account \eqref{EstimError} and \eqref{EstimError2}, and integrating by parts one obtains that
\begin{multline*}
  \frac{d}{dt} E(t) = \int_0^1 \Big( e_t e_{tt} + k e_{x} e_{tx} \Big) \, dx = 
  \int_0^1 \Big( e_t \big( k e_{xx} - \beta \sin z + \beta \sin \widehat{z} \big) - k e_t e_{xx} \Big) \, dx \\
  + k e_x(t, 1) e_t(t, 1) - k e_x(t, 0) e_t(t, 0) = - \alpha k e_t(t, 1)^2 +
  \beta \int_0^1 e_t \big( \sin \widehat{z} - \sin z \big) \, dx.
\end{multline*}
Applying the fact that the function $y \to \sin y$ is globally Lipschitz continuous with the Lipschitz constant $L = 1$
one gets that
$$
  \Big| \int_0^1 e_t \big( \sin \widehat{z} - \sin z \big) \, dx \Big| \le \int_0^1 |e_t| |e| \, dx \le
  \frac{1}{2} \int_0^1 e_t^2 \, dx + \frac{1}{2} \int_0^1 e^2 \, dx.
$$
Hence with the use of Wirtinger's inequality (see, e.g.,~\cite{Hardy}) one obtains that
\begin{equation} \label{ErrorLyapFuncDer_1}
  \frac{d}{dt} E(t) \le - \alpha k e_t(t, 1)^2 + \frac{\beta}{2} \int_0^1 e_t^2 \, dx +
  \frac{2 \beta}{\pi^2} \int_0^1 e_x^2 \, dx.
\end{equation}
Now, let us consider the second term in \eqref{ErrorLyapFuncDer}. At first, note that
\begin{equation} \label{ErrorLyapFuncDer_2}
  \int_0^1 x e_t e_{tx} \, dx = \frac{1}{2}\int_0^1 \Big( \frac{d}{dx} \big( x e_t^2 \big) - e_t^2 \Big) \, dx =
  \frac{1}{2} e_t(t, 1)^2 - \frac{1}{2} \int_0^1 e_t^2 \, dx
\end{equation}
At second, applying \eqref{EstimError} and \eqref{EstimError2} one finds that
\begin{multline} \label{ErrorLyapFuncDer_3}
  \int_0^1 x e_x e_{tt} \, dx = \int_0^1 x e_x \Big( k e_{xx} - \beta \sin z + \beta \sin \widehat{z} \Big) \, dx \\
  \le \frac{k}{2} \int_0^1 \Big( \frac{d}{dx} \big( x e_x^2 \big) - e_x^2 \Big) \, dx +
  \beta \int_0^1 x |e_x| |e| \, dx \\
  \le - \frac{k}{2} \int_0^1 e_x^2 \, dx + \frac{k}{2} e_x^2(t, 1) + 
  \frac{\beta}{2} \int_0^1 e_x^2 \, dx + \frac{\beta}{2} \int_0^1 e^2 \, dx \\
  \le \left(- \frac{k}{2} + \frac{\beta}{2} + \frac{2 \beta}{\pi^2} \right) \int_0^1 e_x^2 \, dx + 
  \frac{k}{2} \alpha^2 e_t^2 (t, 1).
\end{multline}
Combining \eqref{ErrorLyapFuncDer}--\eqref{ErrorLyapFuncDer_3} one gets that for any $t \in [0, T_{\max})$ the following
inequality holds true
\begin{multline*}
  \frac{d}{dt} V(t) \le \left( - \frac{\varepsilon}{2} + \frac{\beta}{2} \right) \int_0^1 e_t^2 \, dx +
  \left( - \frac{\varepsilon k}{2} + \frac{2 \beta}{\pi^2} + \frac{\varepsilon \beta}{2} + 
  \frac{2 \varepsilon \beta}{\pi^2} \right) \int_0^1 e_x^2 \, dx \\
  + \Big( - \alpha k + \frac{\varepsilon}{2} + \frac{\varepsilon k}{2} \alpha^2 \Big) e_t^2(t, 1).
\end{multline*}
Denote
\begin{gather} \label{LinearInequalities}
  \delta_1(\varepsilon) = - \frac{\varepsilon}{2} + \frac{\beta}{2}, \quad
  \delta_2(\varepsilon) = - \frac{\varepsilon k}{2} + \frac{2 \beta}{\pi^2} + \frac{\varepsilon \beta}{2} + 
  \frac{2 \varepsilon \beta}{\pi^2}, \\
  \delta_3(\varepsilon, \alpha) = - \alpha k + \frac{\varepsilon}{2} + \frac{\varepsilon k}{2} \alpha^2.
  \label{LinearInequalities_2}
\end{gather}
Then the above inequality can be rewritten as follows
$$
  \frac{d}{dt} V(t) \le \delta_1(\varepsilon) \int_0^1 e_t^2 \, dx + 
  \delta_2(\varepsilon) \int_0^1 e_x^2 \, dx + \delta_3(\varepsilon, \alpha) e_t^2(t, 1).
$$
Our aim is to show that under the assumptions of the theorem for any $\varepsilon > 0$ and $\alpha > 0$ satisfying the
inequalities
$$
  \eta(\beta, k) < \varepsilon < \min\{ 1, k \}, \quad
  \delta_3(\varepsilon, \alpha) := \frac{\varepsilon k}{2} \alpha^2 - k \alpha + \frac{\varepsilon}{2} \le 0
$$
one has $\delta_1(\varepsilon) < 0$ and $\delta_2(\varepsilon) < 0$. Then for any such $\varepsilon$ and $\alpha$ one
has
$$
  \frac{d}{dt} V(t) \le - \min\left\{ 2 |\delta_1(\varepsilon)|, \frac{2|\delta_2(\varepsilon)|}{k} \right\} E(t).
$$
Consequently, applying \eqref{ErrorLyapFuncEstimates} one gets that
$$
  \frac{d}{dt} V(t) \le - \delta(\varepsilon) V(t), \quad
  \delta(\varepsilon) = \min\left\{ \frac{2|\delta_1(\varepsilon)|}{1 + k_0 \varepsilon}, 
  \frac{2|\delta_2(\varepsilon)|}{k(1 + k_0 \varepsilon)} \right\} > 0,
$$
which yeilds $V(t) \le V(0) e^{- \delta(\varepsilon) t}$ for all $t \in [0, T_{\max})$. Hence and from
\eqref{ErrorLyapFuncEstimates} it follows that
\begin{equation} \label{ExpDecayOfWeightedQuadError}
  E(t) \le \frac{1 + k_0 \varepsilon}{1 - k_0 \varepsilon} E(0) e^{- \delta(\varepsilon) t} 
  \quad \forall t \in [0, T_{\max}),
\end{equation}
which implies the required result.

Thus, it remains to show that for any $\varepsilon > 0$ and $\alpha > 0$ such that
$$
  \eta(\beta, k) < \varepsilon < \min\{ 1, k \}, \quad \delta_3(\varepsilon, \alpha) \le 0
$$
one has $\delta_1(\varepsilon) < 0$ and $\delta_2(\varepsilon) < 0$. From \eqref{LinearInequalities} it follows that
$\varepsilon \in (0, \min\{ 1, k \})$ such that $\delta_1(\varepsilon) < 0$ and $\delta_2(\varepsilon) < 0$ exists if
and only if
$$
  \min\{ 1, k \} > \beta, \quad
  - k + \left( 1 + \frac{4}{\pi^2} \right) \beta < 0, \quad
  \min\{ 1, k \} > \frac{4 \beta}{\pi^2 k - (\pi^2 + 4) \beta}
$$
or, equivalently, $( 1 + 4 / \pi^2 ) \beta < k$ and $\eta(\beta, k) < \min\{ 1, k \}$ (see \eqref{EtaDef}). The above
inequalities are valid due to assumption~\ref{Assumpt_AdmissibleParam}. Therefore for any $\varepsilon > 0$ such that
$\eta(\beta, k) < \varepsilon < \min\{ 1, k \}$ one has $\delta_1(\varepsilon) < 0$  and $\delta_2(\varepsilon) < 0$.
Choosing $\alpha > 0$ such that $\delta_3(\varepsilon, \alpha) \le 0$ (note that such $\alpha$ exists, if 
$\varepsilon < k$; see~\eqref{AdmissibleObserverGain}) one obtains the required result.
\end{proof}

\begin{remark}
The same observer as \eqref{Observer}--\eqref{Observer2} was studied in \cite{Fridman2016} for the
uncontrolled system of the form \eqref{SineGordon}--\eqref{SineGordon3}. It should be noted that sufficient conditions
for the exponential decay of the weighted quadratic error $E(t)$ were formulated in \cite{Fridman2016} in terms of the
feasibility of certain LMIs simultaneously depending on system's parameters $k$ and $\beta$, and the observer gain
$\alpha$. In contrast, in this paper \textit{explicit} inequalities on system's parameters $k$ and $\beta$ and the
observer gain $\alpha$ ensuring the exponential decay of $E(t)$ are obtained. The region of admissible (i.e. satisfying
assumption~\ref{Assumpt_AdmissibleParam}) parameters $k$ and $\beta$ is shown on Fig.~\ref{Fig_AdmissibleRegion}.
\begin{figure}[htpb]
\centering
\includegraphics[width=0.9\linewidth]{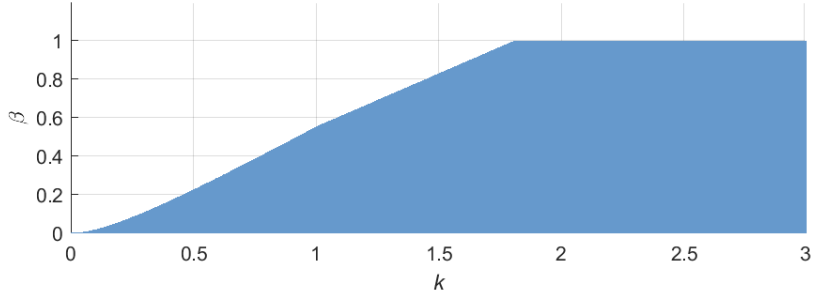}
\caption{The region of admissible parameters $k$ and $\beta$.}
\label{Fig_AdmissibleRegion}
\end{figure}
\noindent{}Note that for any $k$ and $\beta$ from this region one can choose the observer gain $\alpha$ for
which the weighted quadratic error $E(t)$ decays exponentially.
\end{remark}

\subsection{Global Well-Posedness of the Closed-Loop System}

At the second step, let us show that a solution $(z(t), \widehat{z}(t))$ of the system
\eqref{CLSyst_1}--\eqref{CLSyst_2} is defined on $[0, + \infty)$, i.e. $T_{\max} = + \infty$. With the use of
this result we will show that $|H(\widehat{z}(t)) - H(z(t))| \to 0$ as $t \to \infty$.

\begin{theorem} \label{Thr_GlobalWellPosed}
Let Assumptions \ref{Assumpt_WellPosedness} and \ref{Assumpt_AdmissibleParam} be valid. Then for any $\varepsilon > 0$
such that $\eta(\beta, k) < \varepsilon < \min\{ 1, k \}$, for all $\alpha > 0$ satisfying
\eqref{AdmissibleObserverGain}, for all $\gamma > 0$, and for any initial conditions $(z^0, z^1) \in \mathcal{W}_0$ and
$(\widehat{z}^0, \widehat{z}^1) \in \mathcal{W}_0$ there exists $H_{\max} > 0$ such that 
\begin{equation} \label{SolutBounded}
  H(z(t)) \le H_{\max}, \quad H(\widehat{z}(t)) \le H_{\max} 
  \quad \forall t \in [0, T_{\max}),
\end{equation}
which implies that the system is forward complete $(\text{i.e.~} T_{\max} = + \infty)$.
\end{theorem}

\begin{proof}
Denote by $\| \cdot \|$ the standard norm in $L_2(0, 1)$, and let $(z(t), \widehat{z}(t))$ be a solution of
\eqref{CLSyst_1}--\eqref{CLSyst_2}. For any $t \in [0, T_{\max})$ one has
$$
  \| z_t(t) \| \le \| z_t(t) - \widehat{z}_t(t) \| + \| \widehat{z}_t(t) \| \le
  \sqrt{2 E(t)} + \sqrt{2 H(\widehat{z}(t))}.
$$
Hence and from Theorem~\ref{Th_Observer} it follows that there exist $M > 0$ and $\delta > 0$ such that for any
$t \in [0, T_{\max})$ the following inequality holds true:
$$
  \| z_t(t) \|^2 \le 4 E(t) + 4 H(\widehat{z}(t)) \le 4 M E(0) e^{- \delta t} + 4 H(\widehat{z}(t)).
$$
Similarly, for any $t \in [0, T_{\max})$ one has
$$
  \| z_x(t) \|^2 \le \frac{4}{k} M E(0) e^{- \delta t} + \frac{4}{k} H(\widehat{z}(t)).
$$
Therefore for all $t \in [0, T_{\max})$ one has
\begin{equation} \label{PlantEnergyEstim}
  H(z(t)) \le \frac{1}{2} \| z_t(t) \|^2 + \frac{k}{2} \| z_x(t) \|^2 + 2 \beta \le C_1 + C_2 H(\widehat{z}(t)),
\end{equation}
where $C_1 = 2 \beta + 4 M E(0)$ and $C_2 = 4$. Arguing in the same way one can easily verify that
\begin{equation} \label{ObservEnergyEstim}
  H(\widehat{z}(t)) \le C_1 + C_2 H(z(t)) \quad \forall t \in [0, T_{\max}).
\end{equation}
Let $t_0 \in [0, T_{\max})$ be arbitrary. If $H(\widehat{z}(t_0)) \le H^*$, then $H(z(t_0)) \le C_1 + C_2 H^*$
due to \eqref{PlantEnergyEstim}. Suppose, now, that $H(\widehat{z}(t_0)) > H^*$. If $H(\widehat{z}(t)) > H^*$ for all 
$t \in [0, t_0]$, then $H(z(t)) \le H(z(0))$ for any $t \in [0, t_0]$ and $H(\widehat{z}(t_0)) \le C_1 + C_2 H(z(0))$
due to \eqref{ObservEnergyEstim} and the fact that
\begin{equation} \label{DerivOfPlantEnergy}
  \frac{d}{dt} H(z(t)) = - \gamma k \psi( H(\widehat{z}(t)) - H^* ) z_t(t, 1)^2 \le 0
\end{equation}
for all $t \in [0, t_0)$ (see the definition of $\psi(s)$).

On the other hand, if $H(\widehat{z}(t_0)) > H^*$, but there exists $t \in [0, t_0]$ such that
$H(\widehat{z}(t_0)) \le H^*$, then denote $\tau = \sup\{ t \in [0, t_0] \mid H(\widehat{z}(t)) = H^* \}$. Observe that 
$\tau$ is correctly defined and $\tau < t_0$ due to the fact that the function $H(\widehat{z}(t))$ is continuous by our
assumptions on the well-posedness of the closed-loop system. Furthermore, for any $t \in [\tau, t_0]$ one has 
$H(\widehat{z}(t)) > H^*$. Taking into account \eqref{PlantEnergyEstim}, \eqref{DerivOfPlantEnergy} and the definition
of $\tau$ one gets that $H(z(t)) \le H(z(\tau)) \le C_1 + C_2 H^*$ for any $t \in [\tau, t_0]$, which with the use of
\eqref{ObservEnergyEstim} implies that
$$
  H(\widehat{z}(t_0)) \le C_1 + C_2 H(z(t_0)) \le C_1 + C_1 C_2 + C_2^2 H^*.
$$
Since $t_0 \in [0, T_{\max})$ was chosen arbitrarily, one obtains that $H(z(t)) \le H_{\max}$
and $H(\widehat{z}(t)) \le H_{\max}$ for all $t \in [0, T_{\max})$, where
$$
  H_{\max} = \max\Big\{ H^*, H(z(0)), C_1 + C_2 H^*, C_1 + C_2 H(z(0)), C_1 + C_1 C_2 + C_2^2 H^* \Big\}.
$$
It remains to note that the boundedness of $H(z(t))$ and $H(\widehat{z}(t))$ implies that $T_{\max} = + \infty$ by
virtue of our assumption on the well-posedness of the closed-loop system.
\end{proof}

\begin{corollary} \label{Crlr_EnergyObserver}
Let Assumptions \ref{Assumpt_WellPosedness} and \ref{Assumpt_AdmissibleParam} be valid. Then for any $\varepsilon > 0$
such that $\eta(\beta, k) < \varepsilon < \min\{ 1, k \}$, for all $\alpha > 0$ satisfying
\eqref{AdmissibleObserverGain}, for all $\gamma > 0$, and for any initial conditions $(z^0, z^1) \in \mathcal{W}_0$ and
$(\widehat{z}^0, \widehat{z}^1) \in \mathcal{W}_0$ there exist $C$ and $\delta > 0$ such that 
\begin{equation} \label{EnergyDiffEstimate}
  \big| H(z(t)) - H(\widehat{z}(t)) \big| \le C \sqrt{E(0)} e^{- \delta t / 2} 
  \quad \forall t \ge 0.
\end{equation}
\end{corollary}

\begin{proof}
Let $H_{\max} > 0$ be from the theorem above. Then for any $t \ge 0$ one has
$$
  \max\Big\{ \| z_t(t) \|, \| z_x(t) \|, \| \widehat{z}_t(t) \|, \| \widehat{z}_x(t) \| \Big\} \le 
  C_1 := \sqrt{ \max\left\{ 2, \frac{2}{k} \right\} H_{\max}}.
$$
Hence applying the fact that the function $f(s) = s^2 / 2$ is Lipschitz continuous on $[-C_1, C_1]$ with the Lipschitz
constant $L = C_1$ one obtains that
\begin{equation} \label{TimeDerivDiffEst}
  \left| \frac{1}{2} \| z_t(t) \|^2 - \frac{1}{2} \| \widehat{z}_t(t) \|^2 \right| \le 
  C_1 \| z_t(t) - \widehat{z}_t(t) \| \le C_1 \sqrt{2 E(t)} \quad \forall t \ge 0,
\end{equation}
and
\begin{equation} \label{SpaceDerivDiffEst}
  \left| \frac{k}{2} \| z_x(t) \|^2 - \frac{k}{2} \| \widehat{z}_x(t) \|^2 \right| \le 
  C_1 k \| z_x(t) - \widehat{z}_x(t) \| \le C_1 \sqrt{\frac{2}{k} E(t)} \quad \forall t \ge 0.
\end{equation}
Note that since $z(t, 0) = \widehat{z}(t, 0) = 0$, for any $x \in (0, 1)$ one has
$$
  | z(t, x) - \widehat{z}(t, x) | = 
  \Big| \int_0^x \big( z_x(t, y) - \widehat{z}_x(t, y) \big) \, dy \Big| \le
  \| z_x(t) - \widehat{z}_x(t) \|.
$$
Hence and from the fact that the function $f(s) = 1 - \cos(s)$ is Lipschitz continuous with the Lipschitz
constant $L = 1$ it follows that
\begin{multline*}
  \Big| \int_0^1 \big( \beta (1 - \cos z(t, x) ) - \beta (1 - \cos \widehat{z}(t, x)) \big) \, dx \Big| \\
  \le \beta \int_0^1 |z(t, x) - \widehat{z}(t, x)| \, dx \le \beta \| z_x(t) - \widehat{z}_x(t) \| \le
  \beta \sqrt{\frac{2}{k} E(t)}.
\end{multline*}
Combining \eqref{TimeDerivDiffEst}, \eqref{SpaceDerivDiffEst} and the inequality above one gets that there exists 
$C_2 > 0$ such that $| H(z(t)) - H(\widehat{z}(t)) | \le C_2 \sqrt{E(t)}$ for all $t \ge 0$. Now, applying
Theorem~\ref{Th_Observer} we arrive at the required result.
\end{proof}

\subsection{Performance of the Control System}

Now, we are ready to prove that control law \eqref{ControlLaw} indeed solves the energy control problem
\eqref{ControlGoal}.

\begin{theorem}
Let Assumptions \ref{Assumpt_WellPosedness} and \ref{Assumpt_AdmissibleParam} be valid. Then for all $H^* > 0$ and
$\gamma > 0$, for any $\varepsilon > 0$ such that  $\eta(\beta, k) < \varepsilon < \min\{ 1, k \}$, for all $\alpha > 0$
satisfying \eqref{AdmissibleObserverGain}, and for any initial conditions $(z^0, z^1) \in \mathcal{W}_0$ and
$(\widehat{z}^0, \widehat{z}^1) \in \mathcal{W}_0$ such that $H(z(0)) \ne 0$ and $H(z(t)) \ne 0$ for all $t \ge 0$ one
has $H(z(t)) \to H^*$ as $t \to \infty$.
\end{theorem}

\begin{proof}
For any $\varepsilon > 0$ introduce the Lyapunov-like function
$$
  V(t) = H(z(t)) + \varepsilon \sign(H(z(t)) - H^*) g(t), \quad 
  g(t) = \int_0^1 x z_t z_x \, dx.
$$
With the use of the inequalities
$$
  \Big| \int_0^1 x z_t z_x \, dx \Big| \le \int_0^1 |z_t z_x| \, dx \\
  \le \frac{1}{2} \int_0^1 z_t^2 \, dx + \frac{1}{2} \int_0^1 z_x^2 \, dx,
$$
one gets that
\begin{equation} \label{EstimLyapLikeFunc}
  0 \le (1 - \varepsilon k_0) H(z(t)) \le V(t) \le (1 + \varepsilon k_0) H(z(t))
\end{equation}
for all $t \ge 0$ and $\varepsilon < \min\{ 1, k \}$, where $k_0 = \max\{ 1, 1 / k \}$.

Fix arbitrary $t \ge 0$ such that $H(z(t)) \ne H^*$. One has
$$
  \frac{d}{dt} g(t) = \int_0^1 \big( x z_{tt} z_x + x z_t z_{tx} \big) \, dx
  = \int_0^1 \big( x (k z_{xx} - \beta \sin z) z_x + x z_t z_{tx} \big) \, dx.
$$
Note that
$$
  \int_0^1 x z_t z_{tx} \, dx = \frac12 \int_0^1 x (z_t^2)_x \, dx
  = \frac12 z_t(t, 1)^2 - \frac12 \int_0^1 z_t^2 \, dx
$$
and
$$
  \int_0^1 x z_{xx} z_x \, dx = \frac12 \int_0^1 x (z_x^2)_x \, dx
  = \frac12 z_x(t, 1)^2 - \frac12 \int_0^1 z_x^2 \, dx.
$$
Observe also that
\begin{multline*}
  - \int_0^1 x \sin z z_x \, dx = \int_0^1 x (\cos z)_x \, dx = \cos z(t, 1) - \int_0^1 \cos z \, dx \\
  \le \int_0^1 (1 - \cos z) \, dx \le \frac12 \int_0^1 z^2 \, dx \le \frac{2}{\pi^2} \int_0^1 z_x^2 \, dx.
\end{multline*}
Here we used Wirtinger's inequality (see~\cite{Hardy}). Hence for any $\sigma > 0$ one has
\begin{multline*}
  \frac{d}{dt} g(t) \le - \int_0^1 \left( \frac{z_t^2}{2} + k \frac{z_x^2}{2} \right) \, dx
  + \frac12 z_t^2(t, 1) + \frac{k}{2} z_x^2(t, 1) \\
  + \beta \int_0^1 (1 - \cos z) \, dx \le - \frac12 \int_0^1 z_t^2 \, dx -
  \frac{k}{2} \left( 1 - \frac{4 (1 + \sigma) \beta}{k \pi^2} \right) \int_0^1 z_x^2 \, dx \\
  - \sigma \beta \int_0^1 (1 - \cos z) \, dx + \frac12 z_t^2(t, 1) + \frac{k}{2} z_x^2(t, 1).
\end{multline*}
Since $0 \le \beta < k \pi^2 / 4$ by virtue of Assumption~\ref{Assumpt_AdmissibleParam}, there exists $\sigma > 0$ such that
$$
  C_0 := \min\left\{ \sigma, 1 - \frac{4(1 + \sigma) \beta}{k \pi^2} \right\} > 0.
$$
Thus, one gets
\begin{equation} \label{DerivHigherEnergyLevel}
  \frac{d}{dt} g(t) \le - C_0 H(z(t)) + \frac12 z_t^2(t, 1) + \frac{k}{2} z_x^2(t, 1).
\end{equation}
Observe also that 
\begin{equation} \label{LyapFuncDeriv_Smooth}
  \frac{d}{dt} V(t) = - \gamma k \psi(H(\widehat{z}(t)) - H^*) z_t^2(t, 1) + 
  \varepsilon \sign(H(z(t)) - H^*) \frac{d}{dt} g(t)
\end{equation}
for any $t \ge 0$ such that $H(z(t)) \ne H^*$.

Fix an arbitrary $\Delta > 0$. Let us show that for any $T > 0$ there exists $t \ge T$ such that
$|H(z(t)) - H^*| < \Delta$. Indeed, arguing by reductio ad absurdum, suppose that there exists $T > 0$ such that
for any $t \ge T$ one has $|H(z(t)) - H^*| \ge \Delta$. Let us first consider the case when $H(z(t)) \ge H^* + \Delta$
for all $t \ge T$.

From Corollary~\ref{Crlr_EnergyObserver} it follows that there exists $\tau \ge T$ such that
$$
  \big| H(z(t)) - H(\widehat{z}(t)) \big| \le \frac{\Delta}{2} \quad \forall t \ge \tau,
$$
which implies that $H(\widehat{z}(t)) \ge H^* + \Delta / 2$ for all $t \ge \tau$. Define
\begin{gather*}
  \psi_{\Delta} = 
  \min\left\{ \psi(s) \Bigm| s \in \left[ \frac{\Delta}{2}, K \right] \right\} > 0, \\
  \Psi_{\Delta} =
  \max\left\{ \psi(s) \Bigm| s \in \left[ \frac{\Delta}{2}, K \right] \right\} < + \infty,
\end{gather*}
where $K > 0$ is a sufficiently large constant such that $H(\widehat{z}(t)) \le H^* + K$ for all $t \ge 0$
that exists by virtue of Theorem~\ref{Thr_GlobalWellPosed}. By the definitions of $\tau$, $\psi_{\Delta}$ and
$\Psi_{\Delta}$ one has $0 < \psi_{\Delta} \le \psi( H(\widehat{z}(t)) - H^* ) \le \Psi_{\Delta} < + \infty$ 
for all $t \ge \tau$. Applying \eqref{DerivHigherEnergyLevel}, \eqref{LyapFuncDeriv_Smooth} and the above inequality one
obtains that for any $t \ge \tau$ the following inequality holds true
$$
  \frac{d}{dt} V(t) \le - \varepsilon C_0 H(z(t)) - \gamma k \psi_{\Delta} z_t^2(t, 1) +
  \frac{\varepsilon}{2} z_t^2(t, 1) + \frac{\varepsilon k}{2} \gamma^2 \Psi_{\Delta}^2 z_t^2(t, 1).
$$
Choosing $\varepsilon > 0$ sufficiently small and taking into account \eqref{EstimLyapLikeFunc} one gets that
$$
  \frac{d}{dt} V(t) \le - \varepsilon C_0 H(z(t)) \le - C_{\varepsilon} V(t) \quad \forall t \ge \tau,
$$
where $C_{\varepsilon} = \varepsilon C_0 / (1 + \varepsilon k_0)$. 
Therefore $V(t) \le V(\tau) e^{- C_{\varepsilon} (t - \tau)}$ for any $t \ge \tau$, which, due to
\eqref{EstimLyapLikeFunc}, implies that
$$
  H(z(t)) \le \frac{V(\tau)}{1 - \varepsilon k_0} e^{- C_{\varepsilon} (t - \tau)} \quad \forall t \ge \tau.
$$
Thus, $H(z(t)) \to 0$ as $t \to \infty$, which contradicts our assumption that $H(z(t)) \ge H^* + \Delta$ for 
any $t \ge T$. 

Suppose, now, that $H(z(t)) < H^* - \Delta$ for all $t \ge T$. Arguing in a similar way to the case 
$H(z(t)) > H^* + \Delta$, and applying \eqref{DerivHigherEnergyLevel} and \eqref{LyapFuncDeriv_Smooth} one can
verify that for any sufficiently small $\varepsilon > 0$ there exist $\tau \ge T$ and $C_{\varepsilon} > 0$ such that
$V(t) \ge V(\tau) e^{C_{\varepsilon} (t - \tau)}$ for any $t \ge \tau$. Hence with the use of \eqref{EstimLyapLikeFunc}
one gets that
$$
  H(z(t)) \ge \frac{1 - \varepsilon k_0}{1 + \varepsilon k_0} H(z(\tau)) e^{- C_{\varepsilon} t} 
  \quad \forall t \ge \tau.
$$
By our assumption $H(z(\tau)) > 0$. Consequently, $H(z(t)) \to + \infty$ as $t \to \infty$, which contradicts 
the assumption that $H(z(t)) < H^* - \Delta$ for all $t \ge T$. Thus, for any $T > 0$ there exists $t \ge T$ such that 
$|H(z(t)) - H^*| < \Delta$.

Let $\tau > 0$ be such that $| H(z(t)) - H(\widehat{z}(t)) | \le \Delta / 2$ for all $t \ge \tau$ 
(see~Corollary~\ref{Crlr_EnergyObserver}). As we have just proved, there exists $t_0 \ge \tau$ such that 
$|H(z(t_0)) - H^*| < \Delta$. Let us verify that
\begin{equation} \label{EnergyConvergence}
  \big| H(z(t)) - H^* \big| \le \Delta \quad \forall t \ge t_0.
\end{equation}
Then one can conclude that $H(z(t)) \to H^*$ as $t \to \infty$. Arguing by reductio ad absurdum, suppose that there
exists $T \ge t_0$ such that $|H(z(T)) - H^*| > \Delta$. Let us consider the case $H(z(T)) > H^* + \Delta$ first. Define
$$
  \theta = \sup\Big\{ t \in [t_0, T] \Bigm| H(z(t)) = H^* + \Delta \Big\}.
$$
Note that $\theta$ is correctly defined, and $\theta \in (t_0, T)$, since $H(z(t))$ is continuous by our assumption.
From the definition of $\tau$ and $\theta$ it follows that $H(\widehat{z}(t)) \ge H^* + \Delta / 2$ for any 
$t \in [\theta, T]$. Therefore
$$
  \frac{d}{dt} H(z(t)) = - \gamma k \psi\big( H(\widehat{z}(t)) - H^* \big) z_t^2(t, 1) \le 0 
  \quad \forall t \in [\theta, T].
$$
Consequently, $H(z(T)) \le H(z(\theta)) = H^* + \Delta$, which condtradicts the definition of $T$.

Suppose, now, that $H(z(T)) < H^* - \Delta$. Define
$$
  \theta = \sup\Big\{ t \in [t_0, T] \Bigm| H(z(t)) = H^* - \Delta \Big\}.
$$
Taking into account the definition of $\tau$ and the fact that $t_0 \ge \tau$ one gets that 
$H(\widehat{z}(t)) \le H^* - \Delta / 2$ for any $t \in [\theta, T]$, which implies that
$$
  \frac{d}{dt} H(z(t)) = - \gamma k \psi\big( H(\widehat{z}(t)) - H^* \big) z_t^2(t, 1) \ge 0 
  \quad \forall t \in [\theta, T].
$$
Therefore, $H(z(T)) \ge H(z(\theta)) = H^* - \Delta$, which contradicts the definition of $T$. Thus,
\eqref{EnergyConvergence} holds true, and $H(z(t)) \to H^*$ as $t \to \infty$ due to the fact that $\Delta > 0$ was
chosen arbitrarily.
\end{proof}

\begin{remark}
In the theorem above we utilized the assumption that $H(z(t)) \ne 0$  for all $t \ge 0$ (which, in essence, means
that $z(0) \ne 0 \implies z(t) \ne 0$ for all $t \ge 0$). With the use of Corollary~\ref{Crlr_EnergyObserver} one can
verify that this assumption is satisfied, in particular, if the initial conditions of the observer 
$(\widehat{z}^0, \widehat{z}^1)$ are sufficiently close to the initial conditions $(z^0, z^1)$. However, both of these
assumptions seem artificial. Note that the implication $H(z(0)) \ne 0 \implies H(z(t)) \ne 0$ for all $t \ge 0$ can be
easily established if, in particular, solutions of the closed-loop system \eqref{CLSyst_1}--\eqref{CLSyst_2} are
locally unique. Furthermore, even if the local existence and uniqueness theorem is not available (which is the case), it
seems unnatural to expect the control law $u(t) = - \gamma \psi( H(\widehat{z}(t)) - H^* ) z_t(t, 1)$, with
$H(\widehat{z}(t))$ being bounded, to steer the state of the system \eqref{SineGordon}--\eqref{SineGordon3} to the
origin in finite time. However, the authors were unable to prove this result rigorously.
\end{remark}

\section{Numerical evaluation results}\label{Sec_Simulation}

The closed-loop energy control system with plant model \eqref{SineGordon1}--\eqref{SineGordon3}, observer
\eqref{Observer}--\eqref{Observer2} and controller \eqref{ControlLaw} was numerically studied by the simulation in
MATLAB/Simulink software environment. The computation method and simulation results are described below.

\subsection{Computation method}

The PDE equations of the plant model \eqref{SineGordon}--\eqref{SineGordon3} and observer
\eqref{Observer}--\eqref{Observer2} are approximately represented as ODE systems by discretization on a
spatial variable $x$ and implemented by two separate Simulink blocks. 

Consider the discretization procedure for plant model \eqref{SineGordon1}--\eqref{SineGordon3}. The procedure for
observer \eqref{Observer}--\eqref{Observer2} is a similar one with the exception of the boundary conditions.

In the numerical study, the partial differential equation  \eqref{SineGordon1} is discretized in the spatial variable
$x \in \mR^1$ by uniformly splitting the segment $[0,1]$ into  $N$  sub-intervals. The resulting system of $N-1$
ordinary differential equations (ODEs) of the second order is solved over a time interval $[0, T]$ by applying the
variable step Runge--Kutta Method \cite{DormandPrince80}, performed   with  the standard MATLAB routine {\sl ode45}.

At the discretization nodes $x_i=i\cdot h$, $i=1,\dots,N-1$, the second-order spatial derivatives of $z(t,x)$ are
approximately computed  as
\begin{align}
\begin{array}{l}
z_{xx}(t,x_i)=\dfrac{z(t,x_{i+1})-2z(t,x_{i})+z(t,x_{i-1})}{h^2},\quad\text{for}~~2\le i\le N-2,\\
z_{xx}(t,x_1)=\dfrac{z(t,x_{2})-2z(t,x_{1})+z(t,x_{0})}{h^2},\\
z_{xx}(t,x_{N-1})=\dfrac{z(t,x_{N})-2z(t,x_{N-1})+z(t,x_{N-2})}{h^2},
\end{array}
\label{xrr}
\end{align}
where  $h=1/N$ is the {\it discretization step}; $z(t,x_0)$ and $z(t,x_{N})$ correspond to the boundary values of the
PDEs and are calculated outside the ODE solver procedure. The value of $z(t,x_0)$ is taken in accordance with
\eqref{SineGordon2} as $z(t,x_0)=z(t, 0) = 0$. The value of $z(t,x_{N})$ includes the boundary control $u(t)$ and, as
follows from \eqref{SineGordon2},  is calculated as $z(t,x_{N}) =z(t,x_{N-1})+h\,u(t)$. The plant output $y(t)$ is found
in accordance with \eqref{SineGordon3} as $y(t)=z_t(t,x_{N-1})$.
The remaining values $z(t,x_{1})$, \dots, $z(t,x_{N-1})$ are computed by numerical solving the following $(N-1)$ ODE
equations of the second order:
\begin{align}
z_{tt}(t,x_i)=kz_{xx}(t,x_i)-\beta \sin\big(z(t,x_i)\big),\quad 1\le i\le N-1
\label{odes}
\end{align}
with the given initial \eqref{SineGordon1} and boundary  \eqref{SineGordon2} conditions.

The similar procedure is used for ODE representation of the observer equation \eqref{Observer} with the exception that
boundary value $\widehat z(t,x_{N})$ is found as $\widehat z(t,x_{N}) =\widehat
z(t,x_{N-1})+h\left(u(t)+\alpha\big(y(t)-\widehat{z}_t(t,x_{N-1}\big)\right)$, see \eqref{Observer2}, and the initial
conditions are defined by \eqref{Observer1}.

Simulink blocks for plant \eqref{SineGordon1}--\eqref{SineGordon3} and observer \eqref{Observer}--\eqref{Observer2}
models are connected by the measurement relation $y(t)=z_t(t,x_{N-1})$ and controller \eqref{ControlLaw} equation.

\subsection{Simulation results}

The following  parameters were used in the simulation: $k = 0.12$, $\beta = 0.02$ and $\alpha = 20$. Initial conditions 
\eqref{SineGordon2} were set to: $z^0(x) = 5 (1 - \cos(2\pi x))$, $z^1(x) = 0$. For observer
\eqref{Observer}--\eqref{Observer2} zero initial conditions were taken. The desirable energy level was chosen as 
$H^* = 10$. The simulation time was confined to $50$. For simulations, the system of PDEs was uniformly discretized in
the spacial variable $x$ on $N = 1000$ intervals.

The simulation results are depicted in Figs.~\ref{eobscool}--\ref{zcool}. The observer behaviour is illustrated by
Figs.~\ref{eobscool}--\ref{zt05cool}. The spatial-temporal plot of the observation error 
$e(t,x) = z(t,x) - \widehat z(t,x)$ is shown in Fig.~\ref{eobscool}.  Fig.~\ref{Etcool} demonstrates observer
weighted quadratic error $E(t)$ time history. As is seen from the plots, the observation error decays exponentially with
the transient time about $30$ time units. 
\begin{figure}[htpb]
\centering
\includegraphics[width=0.9\linewidth]{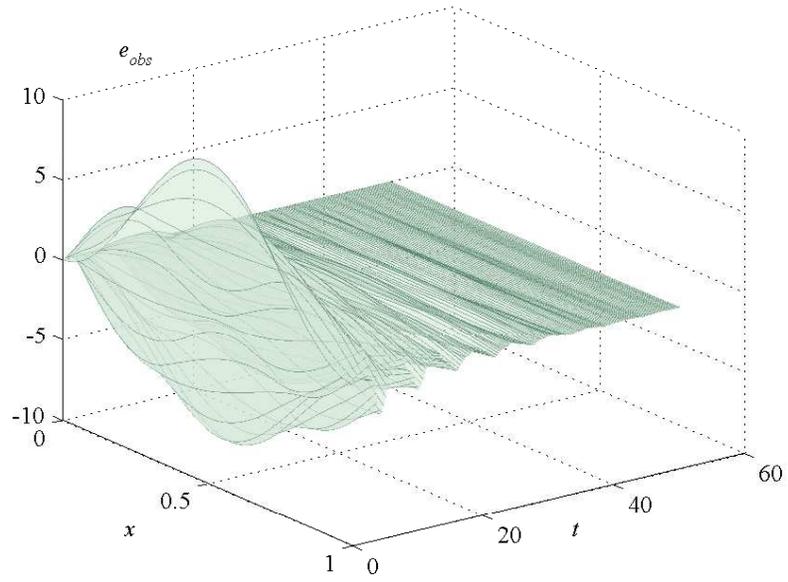}
\caption{The spatial-temporal plot of the observation error $e(t,x)$.}
\label{eobscool}
\end{figure}

\begin{figure}[htpb]
\centering
\includegraphics[width=0.9\linewidth]{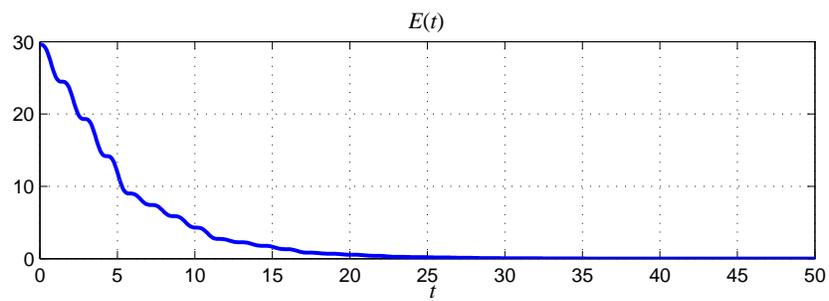}
\caption{Weighted quadratic error $E(t)$ time history.}
\label{Etcool}
\end{figure}

Figures~\ref{z05cool}, \ref{zt05cool} demonstrate the estimation of $z(t,x)$, $z_t(t,x)$ for the particular point
$x = 0.5$ (the middle of the spatial interval $x \in [0,1]$). The plots show that the transient response time of the
observation error is the same.

\begin{figure}[htpb]
\centering
\includegraphics[width=0.9\linewidth]{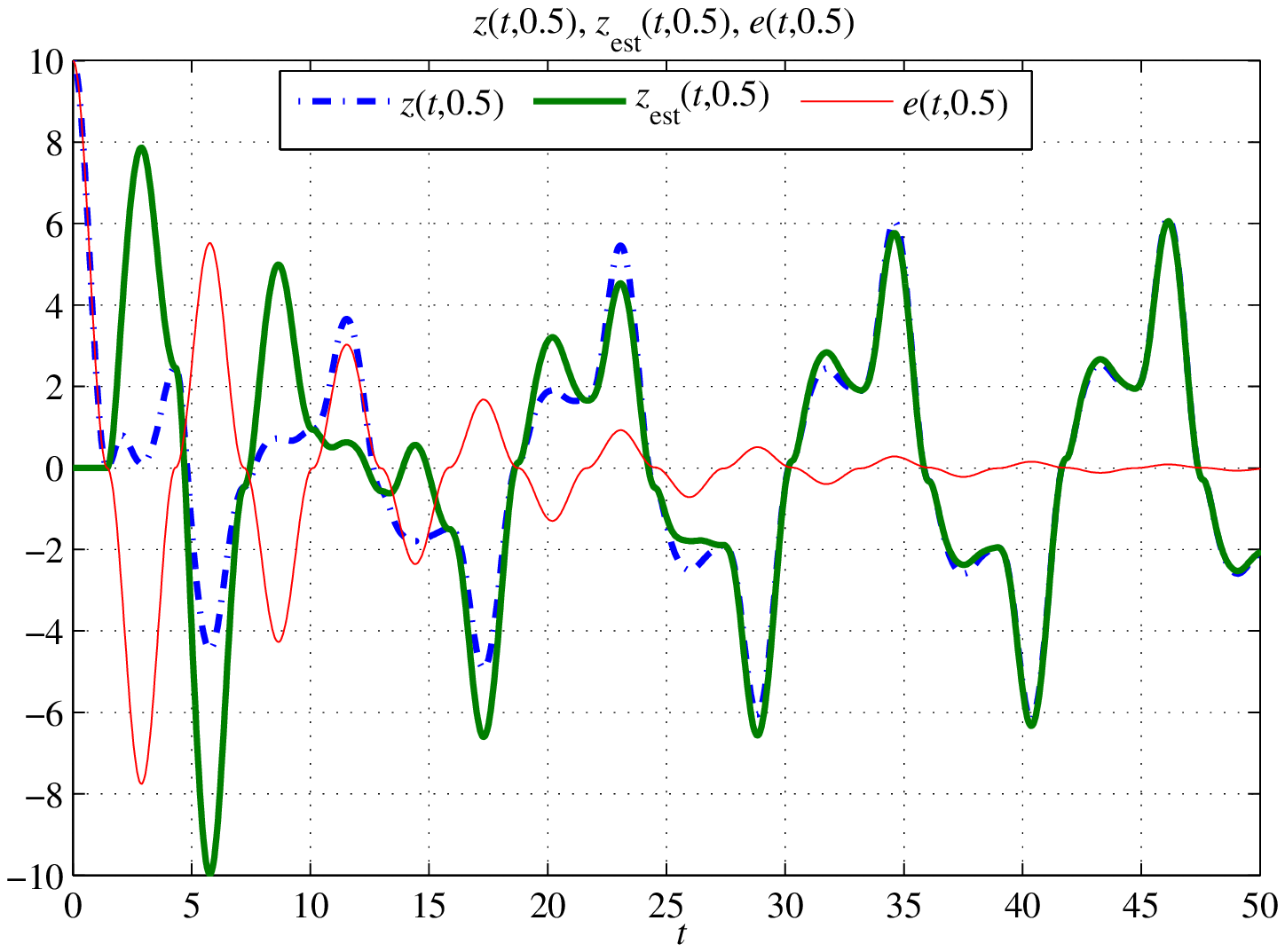}
\caption{Process of $z(t,x)$ estimation for $x=0.5$.  $z(t,0.5)$ and $\hat z(t,0.5)$ time histories.}
\label{z05cool}
\end{figure}
\begin{figure}[htpb]
\centering
\includegraphics[width=0.9\linewidth]{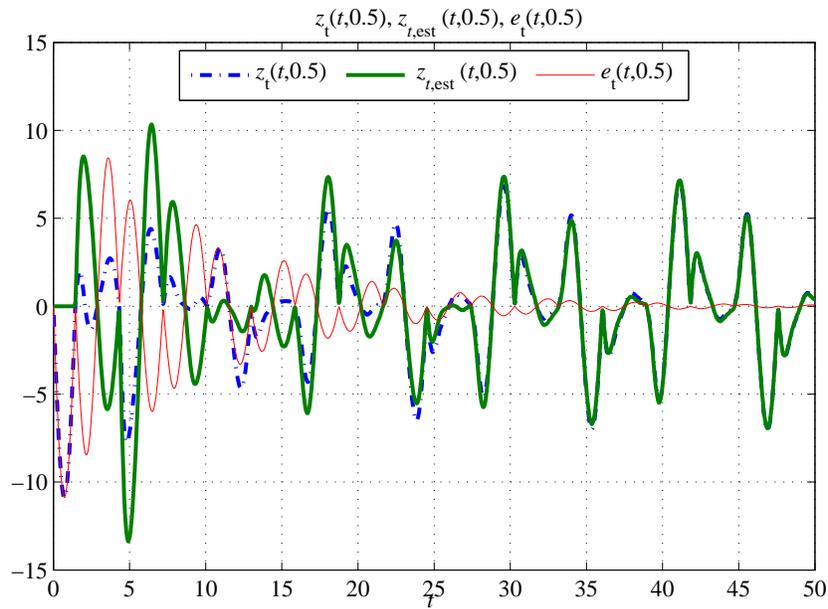}
\caption{Process of $z_t(t,x)$ estimation for $x=0.5$.  $z_t(t,0.5)$ and $\hat {z}_t(t,0.5)$ time histories.}
\label{zt05cool}
\end{figure}

The closed-loop plant-observer-controller \eqref{SineGordon1}--\eqref{SineGordon3},
\eqref{Observer}--\eqref{Observer2},  \eqref{ControlLaw} system behavior is illustrated by Figs.~\ref{Hucool} and
\ref{zcool}. Control action $u(t)$ time history is shown in Fig.~\ref{Hucool} (upper plot), system's energy $H(t)$ and
energy estimate $\widehat H(t)$ time histories are demonstrated in  Fig.~\ref{Hucool}  (lower plot). As is seen from the
$H(t)$ time history, it achieves the prescribed reference value $H^* = 10$ simultaneously with vanishing of the energy
estimation error $H(t) - \widehat H(t)$, and the transient time is about $40$ time units, which is close to that of the
observer.

\begin{figure}[htpb]
\centering
\includegraphics[width=0.9\linewidth]{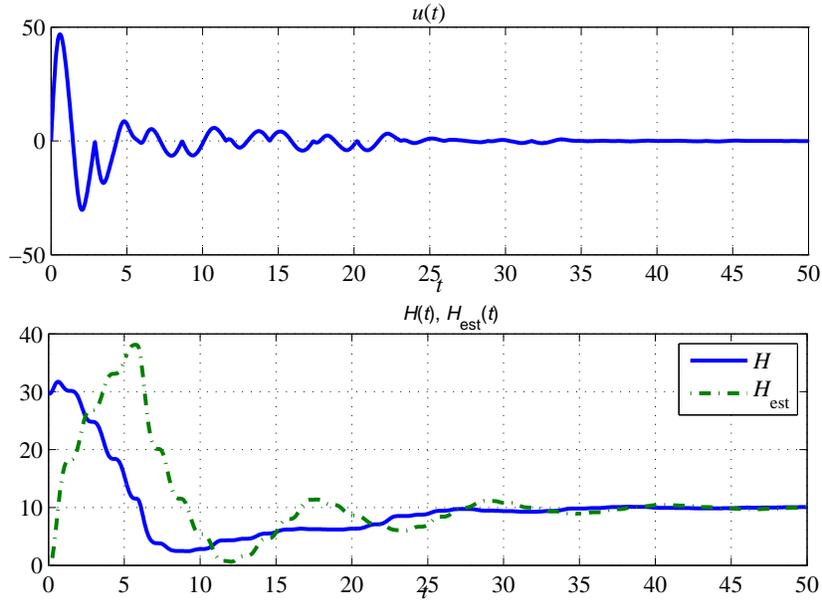}
\caption{Control action $u(t)$ (upper plot), system's energy $H(t)$ and energy estimate $\widehat H(t)$ (lower plot)
time histories.}
\label{Hucool}
\end{figure}

\begin{figure}[htpb]
\centering
\includegraphics[width=0.9\linewidth]{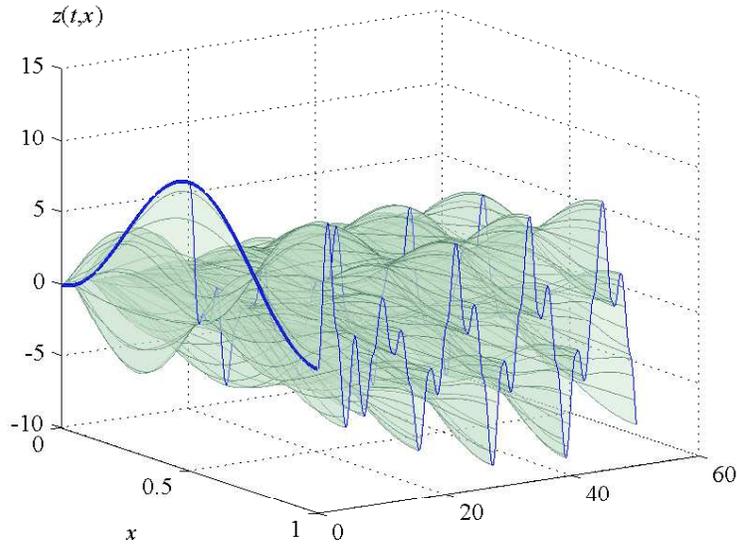}
\caption{The spatial-temporal plot of $z(t,x)$.}
\label{zcool}
\end{figure}

\section{Conclusions}

In this paper the problem of observer-based boundary control of the sine-Gordon model energy is posed for the first
time. A Luenberger-type observer for the sine-Gordon equation is analysed, and explicit inequalities on equation's
parameters ensuring the exponential decay of the estimation error are obtained. With the use of this observer a
speed-gradient control law for solving the energy control problem is proposed. Under the assumption that system's energy
does not vanish in finite time the achievement of the control goal is proved. The results of numerical experiments
demonstrate that the transient time in energy is close to the transient time in observation error, i.e. the closed-loop
system has a reasonable performance.

\section*{References}

\bibliography{IEEEabrv,DFAbib}

\end{document}